\documentclass[11pt,a4paper]{amsart}
\usepackage[utf8]{inputenc}
\usepackage{amsmath}
\usepackage{amsfonts}
\usepackage{amssymb}
\usepackage[dvipsnames]{xcolor}
\usepackage{amsmath,amssymb, amsthm}
\usepackage[smalltableaux]{ytableau}
\usepackage{hyperref}
\usepackage[english]{babel}
\usepackage{mathtools}
\usepackage{geometry}

\allowdisplaybreaks[2] 

\definecolor{lightgray}{gray}{0.8}
\newcommand{\CoTh}{lightgray}  

\theoremstyle{plain}
\newtheorem{theorem}{Theorem}

\newtheorem{lemma}{Lemma}
\newtheorem{corollary}{Corollary}

\newcommand{\Q}{\mathbb{Q}}
\newcommand{\R}{\mathbb{R}}

\newcommand{\lp}{\left(}
\newcommand{\rp}{\right)}
\newcommand{\defeq}{\coloneqq}
\newcommand{\epsi}{\varepsilon}
\newcommand{\boldepsi}{\boldsymbol{\epsi}}
\newcommand{\intpart}[1]{\left\lfloor#1\right\rfloor}
\DeclareMathOperator{\odd}{odd}
\DeclareMathOperator{\even}{even}
\DeclareMathOperator{\Point}{P}
\DeclareMathOperator{\Horizontal}{H}
\DeclareMathOperator{\Vertical}{V}
\DeclareMathOperator{\Square}{S}
\DeclareMathOperator{\new}{new}

\newenvironment{List}{\begin{list}{$\bullet$}{
\setlength{\labelwidth}{.5cm}
\setlength{\leftmargin}{.7cm}
}
}{\end{list}}

\author[F.~Battistoni]{Francesco Battistoni}
\address{Laboratoire de math\'{e}matiques de Besan\c{c}on\\
         Universit\'{e} Bourgogne Franche-Comt\'{e}\\
         CNRS - UMR 6623\\
         16\\ Route de Gray\\
         25030 Besan\c{c}on\\
         France}
\email{francesco.battistoni@univ-fcomte.fr}
\thanks{The first author was supported by the French “Investissements d'Avenir"
program, project ISITE-BFC (contract ANR-15-IDEX-0003)}

\author[G.~Molteni]{Giuseppe Molteni}
\address{Dipartimento di Matematica\\
         Universit\`{a} di Milano\\
         via Saldini 50\\
         20133 Milano\\
         Italy}
\email{giuseppe.molteni1@unimi.it}



\title{An elementary proof for a generalization of a Pohst's inequality}

\begin{document}

\begin{abstract}
Let
\[
P_n(y_1,\ldots,y_n)\defeq \prod_{1\leq i<j\leq n}\lp 1 -\frac{y_i}{y_j}\rp
\]
and
\[
P_n\defeq \sup_{(y_1,\ldots,y_n)}P_n(y_1,\ldots,y_n)
\]
where the supremum is taken over the $n$-ples $(y_1,\ldots,y_n)$ of real numbers satisfying $0 <|y_1| <
|y_2|< \cdots < |y_n|$. We prove that $P_n \leq 2^{\intpart{n/2}}$ for every $n$, i.e., we extend to all
$n$ the bound that Pohst proved for $n\leq 11$. As a consequence, the bound for the absolute discriminant
of a totally real field in terms of its regulator is now proved for every degree of the field.
\end{abstract}

\maketitle

\section{Introduction}
Let $y_1,\ldots,y_n$ be $n\geq 2$ non zero real numbers satisfying the condition
\begin{equation}\label{eq:A1}
  |y_1|< |y_2|< \cdots < |y_n|.
\end{equation}
Define then the positive real number
\[
P_n(y_1,\ldots,y_n)\defeq \prod_{1\leq i<j\leq n}\lp 1 -\frac{y_i}{y_j}\rp
\]
and consider
\[
P_n\defeq \sup_{(y_1,\ldots,y_n)}P_n(y_1,\ldots,y_n),
\]
where the supremum is taken over the $n$-ples of real numbers $(y_1,\ldots,y_n)$ which satisfy the
condition~\eqref{eq:A1}.

The goal of this paper is to provide an estimation for $P_n$ for every $n\geq 2$. This is motivated by
number theoretic reasons: in fact, let $K$ be a number field of degree $n\geq 2$ and let $\epsi$ be a
unit of its ring of integers such that $K=\Q(\epsi)$. The discriminant $d_K$ of the field $K$ divides the
discriminant of the minimum polynomial of $\epsi$, inducing the inequality
\[
|d_K|\leq \prod_{1\leq i < j\leq n}|\epsi_i-\epsi_j|^2
\leq \prod_{k=2}^n |\epsi_k|^{2(k-1)} \cdot \prod_{1\leq i<j\leq n}\lp 1 -\frac{\epsi_i}{\epsi_j}\rp^2
\leq \prod_{k=2}^n |\epsi_k|^{2(k-1)} \cdot P_n^2.
\]
Furthermore, when $K$ is totally real and primitive (i.e. has no proper subfields) it is possible to
estimate the remaining product in terms of the regulator $R_K$ of $K$ with classical methods from
geometry of numbers (for example see~\cite{AstudilloDiaz-y-DiazFriedman}), and one obtains that
\begin{align*}
\log|d_K|
\leq \sqrt{\gamma_{n-1}\cdot\frac{n^3-n}{3}}\cdot(\sqrt{n} R_K)^{1/(n-1)} + 2\log P_n,
\end{align*}
where $\gamma_{n-1}$ denotes the Hermite constant of dimension $n-1$ (for the definition of this constant
see \cite[Ch.~3, Sec.~3]{PohstZassenhaus}). Thus, any estimation for $P_n$ provides an estimation for the
discriminant $d_K$. More precisely, Remak~\cite{Remak} first showed that $P_n \leq n^{n/2}$ for every
$n$; Pohst~\cite{Pohst} improved the bound to $P_n \leq 2^{\intpart{n/2}}$ for every $n\leq 11$, and
Bertin~\cite{Bertin} produced a new proof of Remak's estimate. In the same paper Bertin also gave an
argument trying to prove that Pohst's estimation holds for every $n$, but her procedure is not completely
convincing. In this paper we prove that Pohst's estimation holds indeed for every $n$.

In order to achieve this result, following aforementioned works, we choose a slightly different function
to estimate: given $P_n(y_1,\ldots,y_n)$, we define the change of variables
\begin{equation}\label{eq:A2}
x_i\defeq \frac{y_i}{y_{i+1}},
\qquad
i=1,\ldots,n-1
\end{equation}
which transforms $P_n(y_1,\ldots,y_n)$ into the quantity
\[
Q_{n-1}(x_1,\ldots,x_{n-1}) \defeq \prod_{i=1}^{n-1}\prod_{j=i}^{n-1} \lp 1-\prod_{k=i}^j x_k \rp.
\]
Since $|x_i|\leq 1$ for every $i$, the polynomials $Q_{n}$ are non negative over the cube
$D_n\defeq [-1,1]^n$, and we look for
\[
M_n\defeq \max_{(x_1,\ldots,x_n)\in D_n} Q_n(x_1,\ldots,x_n).
\]
The change of variables~\eqref{eq:A2} shows that $P_n = M_{n-1}$ for every $n\geq 2$. Starting from this,
in the next sections we will prove the following theorem.
\begin{theorem}\label{th:A1}
The maximum $M_n$ of $Q_n$ in $D_n$ is $2^{\intpart{\frac{n+1}{2}}}$ for every $n$, so that $P_n =
M_{n-1} = 2^{\intpart{n/2}}$ for every $n\geq 2$.
\end{theorem}
It is easy to verify that $Q_n$ attains its maximum at $(-1,0,-1,0,\ldots)$ when $n$ is odd, while for an
even $n$ this happens at each point $([-1,0]^k,[0,-1]^{n/2-k})$ for any choice of $k=0,1,\ldots,n/2$
(here $[-1,0]^k$ means that the string $[-1,0]$ has to be repeated $k$ times, the same for
$[0,-1]^{n/2-k}$). Our argument proving Theorem~\ref{th:A1} can be adapted to prove also that these are
the unique points where $Q_n$ attains its maximum, but we leave to the interested reader a formal proof
of this fact.
%
%
%
%
\begin{corollary}
Let $K$ be a totally real and primitive field of degree $n\geq 2$ having discriminant $d_K$ and regulator
$R_K$. Then
\[
\log|d_K| \leq \sqrt{\gamma_{n-1}\cdot \frac{n^3-n}{3}}\cdot(\sqrt{n} R_K)^{1/(n-1)} + \intpart{\frac{n}{2}}\log 4 .
\]
\end{corollary}

\section{Basic inequalities}
An elementary computation shows that the maximums for the first two polynomials $Q_1(x_1)=(1-x_1)$ and
$Q_2(x_1,x_2)=(1-x_1)(1-x_1x_2)(1-x_2)$ are
\begin{equation}\label{eq:A3}
M_1=2,
\qquad
M_2=2,
\end{equation}
respectively. These numbers agree with the claim of the theorem. It is clear that the determination of
$M_n$ via local, i.e. analytic, methods involving partial derivatives becomes quickly infeasible as $n$
increases: we take a different and global, so to say, approach, where the polynomial is split in suitable
blocks and the maximum for the polynomial is deduced from the maximums of those blocks. These maximums
will be deduced from the following basic inequalities.
\begin{lemma}\label{lem:A1}
Let $x,y,z$ be real numbers in $[0,1]$. Then the following inequalities hold:
    \begin{equation}\label{eq:A4}
    (1-x)(1+xy) \leq 1,
    \end{equation}
    \begin{equation}\label{eq:A5}
    (1-x)(1+xy)\leq (1+x)(1-xy),
    \end{equation}
    \begin{equation}\label{eq:A6}
    (1-y)(1+xy)(1+yz)(1-xyz) \leq (1+y)(1-xy)(1-yz)(1+xyz),
    \end{equation}
    \begin{equation}\label{eq:A7}
    (1-y)(1+xy)(1+yz)(1-xyz) \leq 1.
    \end{equation}
\end{lemma}

\begin{proof}
\eqref{eq:A4} is obvious, since $1+xy\leq 1+x$.
\eqref{eq:A5} is reduced via direct computations to $-x+xy\leq 0$, which is clearly true.
For~\eqref{eq:A6}: the right hand side minus the left hand side factorizes as
\begin{align*}
    2y(1-x)(1-z)(1+xy^2z),
\end{align*}
%
which is nonnegative under our hypotheses.\\
Finally, \eqref{eq:A7} already appears in~\cite{Pohst}; for sake of completeness we recall here a quicker
proof. Compute all the products, remove the common terms, factor out $y$ and move the terms to left hand
side or right hand side according to the sign of the coefficient. In this way the inequality is proved to
be equivalent to
\begin{align*}
(y^3 z^2 + y^2z)x^2 + (y^2z^2 + yz   + 1    )x + xyz + z
\leq
(y^2 z^2 + y  z)x^2 + (y^2z   + yz^2 + y + z)x +  yz + 1.
\end{align*}
%
%
This inequality is true since it can be obtained adding the three inequalities
\begin{align*}
x^2y^3 z^2 + x^2y^2z + xy^2z^2
&\leq
x^2y^2 z^2 + x^2y  z + xy^2z,\\
x +y + z + xyz
&\leq
xy + xz +  yz + 1,\\
xyz   -y
&\leq
xyz^2
\end{align*}
which are true (the first one because each term appearing to the left contains and extra power with respect
to the corresponding term to the right, the second one because it can be written as $(1-x)(1-y)(1-z)\geq
0$, and the last one because $y(1-xz+xz^2)\geq 0$ in the given range).
\end{proof}

\section{Graphical schemes}
We call \emph{graphical scheme of dimension $n$} any triangular $n\times n$ array $C$ with symbols
``$+$'' or ``$-$'' in each entry $C_{i,j}$ with $1\leq i\leq j\leq n$. The following are some examples of
graphical schemes in dimension $n=3$ and $n=5$, respectively:
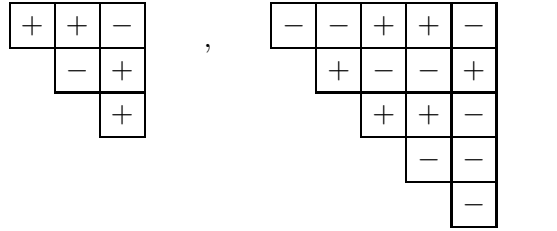
\[
\ytableausetup{nosmalltableaux}
\begin{ytableau}
 {+}   & {+}   & {-}   \\
 \none & {-}   & {+}   \\
 \none & \none & {+}   \\
\end{ytableau}
\qquad
\text{,}
\qquad
\begin{ytableau}
 {-}   & {-}   & {+}   & {+}   & {-}\\
 \none & {+}   & {-}   & {-}   & {+}\\
 \none & \none & {+}   & {+}   & {-}\\
 \none & \none & \none & {-}   & {-}\\
 \none & \none & \none & \none & {-}\\
\end{ytableau}
\qquad
\text{.}
\]
We associate with $C$ the function $F_C\colon [0,1]^n\to \R$ defined as
\[
F_C(z_1,\ldots,z_n)\defeq \prod_{i=1}^n\prod_{j=i}^n \lp 1 -C_{i,j}\prod_{k=i}^j z_k\rp,
\]
and we denote its $(i,j)$ factor as
\[
F_{C_{i,j}}\defeq 1 -C_{i,j}\prod_{k=i}^j z_k.
\]
Given two graphical schemes $C$ and $C'$ of dimension $n$, we say that $C\leq C'$ if
$F_C(z_1,\ldots,z_n)\leq F_{C'}(z_1,\ldots,z_n)$ for every choice of $(z_1,\ldots,z_n)\in [0,1]^n$. The
following lemma describes four basic moves that when performed on a given scheme produce a larger (in the
previous sense) scheme.
\begin{lemma}\label{lem:A2}
Let $C$ be a graphical scheme of dimension $n$.
\begin{List}
\item[P)](\textbf{Point}) Assume $C_{i,j}=+$. Let $C'$ be the graphical scheme defined by
\[
C_{r,s}'= \begin{cases}
                  - & (r,s)=(i,j)\\
            C_{r,s} & \text{otherwise.}
          \end{cases}
\]
Then  $C\leq C'$. Moreover, $F_{C_{i,j}}\leq 1$.
\item[H)](\textbf{Horizontal segment}) Assume $C_{i,j}=+$ and $C_{i,j+k}=-$, with $k\leq n-j$. Let
    $C'$ be the graphical scheme defined by
\[
C_{r,s}'= \begin{cases}
                  - & (r,s)=(i,j)\\
                  + & (r,s)=(i,j+k)\\
            C_{l,k} & \text{otherwise.}
          \end{cases}
\]
Then $C\leq C'$. Moreover, $F_{C_{i,j}}\cdot F_{C_{i,j+k}}\leq 1.$
\item[V)](\textbf{Vertical segment}) Assume $C_{i,j}=-$ and $C_{i+k,j}=+$ with $k\leq j-i$. Let $C'$
    be the graphical scheme defined by
\[
C_{r,s}'= \begin{cases}
                  + & (r,s)=(i,j)\\
                  - & (r,s)=(i+k,j)\\
            C_{l,k} & \text{otherwise.}
            \end{cases}
\]
Then $C\leq C'$. Moreover, $F_{C_{i,j}}\cdot F_{C_{i+k,j}}\leq 1.$
\item[S)](\textbf{Square}) Assume $C_{i,j}=-, C_{i,j+k}=+, C_{i+l,j}=+$ and $C_{i+l,j+k}=-$. Let $C'$
    be the graphical scheme defined by
\[
C_{r,s}'= \begin{cases}
                  + & (r,s)=(i,j)\\
                  - & (r,s)=(i,j+k)\\
                  - & (r,s)=(i+l,j)\\
                  + & (r,s)=(i+l,j+k)\\
            C_{l,k} & \text{otherwise.}
          \end{cases}
\]
Then $C\leq C'$. Moreover, $F_{C_{i,j}}F_{C_{i+l,j}}F_{C_{i,j+k}}F_{C_{i+l,j+k}}\leq 1$.
\end{List}
\end{lemma}
\noindent %
We introduce a notation for these moves:
\begin{List}
\item[P)] {\bf Point}:
$\Point(i;j)$ denotes the change of
$\ytableausetup{smalltableaux,aligntableaux=bottom}
\begin{ytableau}
 \none     & \none[j]\\
 \none[i]  & {+}     \\
\end{ytableau}
$
\ into
$
\begin{ytableau}
 \none     & \none[j]\\
 \none[i]  & {-}     \\
\end{ytableau}
$\ ,
\item[H)] {\bf Horizontal}:
$\Horizontal(i;j,j')$ denotes the change of
$\ytableausetup{smalltableaux,aligntableaux=bottom}
\begin{ytableau}
 \none     & \none[j] & \none[j']\\
 \none[i]  & {+}      & {-}      \\
\end{ytableau}
$
\ into
$
\begin{ytableau}
 \none     & \none[j] & \none[j']\\
 \none[i]  & {-}      & {+}      \\
\end{ytableau}
$\ ,
\item[V)] {\bf Vertical}:
$\Vertical(i,i';j)$ denotes the change of
$\ytableausetup{smalltableaux,aligntableaux=bottom}
\begin{ytableau}
 \none     & \none[j] \\
 \none[i]  & {-}      \\
 \none[i'] & {+}      \\
\end{ytableau}
$
\ into
$
\begin{ytableau}
 \none     & \none[j] \\
 \none[i]  & {+}      \\
 \none[i'] & {-}      \\
\end{ytableau}
$\ ,
\item[S)] {\bf Square}:
$\Square(i,i';j,j')$ denotes the change of
$\ytableausetup{smalltableaux,aligntableaux=bottom}
\begin{ytableau}
 \none     & \none[j] & \none[j']\\
 \none[i]  & {-}      & {+}      \\
 \none[i'] & {+}      & {-}      \\
\end{ytableau}
$
\ into
$
\begin{ytableau}
 \none     & \none[j] & \none[j']\\
 \none[i]  & {+}      & {-}      \\
 \none[i'] & {-}      & {+}      \\
\end{ytableau}
$\ .
\end{List}
\begin{proof}\mbox{}
\begin{List}
\item[P)] We have
\[
F_{C_{i,j}}
=    1-\prod_{k=i}^j z_k
\leq 1+\prod_{k=i}^j z_k
\]
and since every other factor of $F_C$ remains unchanged, we get $F_C\leq F_{C'}$. The statement
$F_{C_{i,j}}\leq 1$ is immediate.
\item[H)] $F_{C_{i,j}}\cdot F_{C_{i,j+k}}\leq 1$ is a direct consequence of~\eqref{eq:A4},
    while~\eqref{eq:A5} implies
\begin{align*}
F_{C_{i,j}}\cdot F_{C_{i,j+k}} &=
\lp 1-\prod_{l=i}^j z_l\rp  \lp 1+\prod_{l=i}^{j}z_l\prod_{l=j+1}^{j+k} z_l\rp
\\
&\leq\lp 1+\prod_{l=i}^j z_l\rp  \lp 1-\prod_{l=i}^{j}z_l\prod_{l=j+1}^{j+k} z_l\rp
= F_{C_{i,j}'}\cdot F_{C_{i,j+k}'}
\end{align*}
and this proves $F_C\leq F_{C'}$ since every other factor is unchanged.
\item[V)] is proved in a similar way to case H).
\item[S)] $F_{C_{i,j}}F_{C_{i+l,j}}F_{C_{i,j+k}}F_{C_{i+l,j+k}}\leq 1$ is a direct application
    of~\eqref{eq:A7}, while~\eqref{eq:A6} implies
\begin{align*}
&F_{C_{i+l,j}}\cdot F_{C_{i,j}}\cdot F_{C_{i+l,j+k}}\cdot F_{C_{i,j+k}}\\
&=\lp 1{{-}}\prod_{v=i{+}l}^j z_v\rp  \lp 1{+}\prod_{v=i}^{i{+}l{-}1}z_v\prod_{v=i{+}l}^{j} z_v\rp
\lp 1{{+}}\prod_{v=i{+}l}^j z_v\prod_{v=j{+}1}^{j{+}k} z_v\rp  \lp 1{-}\prod_{v=i}^{i{+}l{-}1}z_v\prod_{v=i{+}l}^{j} z_v\prod_{v=j{+}1}^{j{+}k} z_v\rp\\
&\leq \lp 1{+}\prod_{v=i{+}l}^j z_v\rp  \lp 1{-}\prod_{v=i}^{i{+}l{-}1}z_v\prod_{v=i{+}l}^{j} z_v\rp
\lp 1{-}\prod_{v=i{+}l}^j z_v\prod_{v=j{+}1}^{j{+}k} z_v\rp  \lp 1{+}\prod_{v=i}^{i{+}l{-}1}z_v\prod_{v=i{+}l}^{j} z_v\prod_{v=j{+}1}^{j{+}k} z_v\rp\\
&=F_{C_{i+l,j}'}\cdot F_{C_{i,j}'}\cdot F_{C_{i+l,j+k}'}\cdot F_{C_{i,j+k}'}
\end{align*}
and this proves $F_C\leq F_{C'}$ since every other factor is unchanged.
\end{List}
\end{proof}

\section{Properties of the schemes generated by sign vectors}
Identifying numbers $\pm 1$ with symbols $\pm$, we can generate a graphical scheme $C(\boldepsi)$ from
each signs vector $\boldepsi\defeq (\epsi_1,\ldots,\epsi_n)$, $\epsi_k\in\{\pm 1\}$, by setting
$C(\boldepsi)_{i,j}\defeq \prod_{k=i}^j \epsi_k$ for every $(i,j)$. For example, the vector
$\boldepsi\defeq(1,1,-1,1,-1)$ generates the scheme
\[
C(\boldepsi)
=
\ytableausetup{nosmalltableaux,centertableaux}
\begin{ytableau}
 {+}   & {+}   & {-}   & {-}   & {+}\\
 \none & {+}   & {-}   & {-}   & {+}\\
 \none & \none & {-}   & {-}   & {+}\\
 \none & \none & \none & {+}   & {-}\\
 \none & \none & \none & \none & {-}\\
\end{ytableau}.
\ytableausetup{smalltableaux}
\]
The interest for this construction comes from the following remark. We can split $D_n=[-1,1]^n$ into
$2^n$ different chambers $D_{n,\boldepsi}$, each one associated with a different signs vector
$\boldepsi=(\epsi_1,\ldots,\epsi_n)$, where
\[
D_{n,\boldepsi}\defeq \{(x_1,\ldots,x_n)\in [-1,1]^n\colon x_i\epsi_i\geq 0,\ \forall i\}.
\]
Once we have chosen $D_{n,\boldepsi}$, the change of variables $z_i\defeq \epsi_i x_i$
transforms $D_{n,\boldepsi}$ into $[0,1]^n$, and $Q_n(x_1,\ldots,x_n)$ into
\[
Q_n(\epsi_1 z_1,\ldots,\epsi_n z_n)
= \prod_{i=1}^n\prod_{j=i}^n \lp 1 -\prod_{k=i}^j\epsi_k\prod_{k=i}^j z_k\rp,
\]
which is exactly the polynomial $F_{C(\boldepsi)}$ associated with the scheme $C(\boldepsi)$ generated
by the signs vector $\boldepsi$.
This gives us a strategy to prove Theorem~\ref{th:A1}: we will prove that for each scheme $C(\boldepsi)$
there is a list of moves $P$, $V$, $H$ and $S$ which transform $C(\boldepsi)$ into $C_-$, the
$n$-dimensional scheme generated by the signs $\boldepsi_-\defeq(-1,\cdots,-1)$ (see next
Theorem~\ref{th:A2}): by Lemma~\ref{lem:A2} these moves increase the value of the associated polynomial,
hence the maximum of each $F_{C(\boldepsi)}$ is lower than the one of $F_{C_-}$. In other words, this
means that the maximum of $Q_n$ in every chamber $D_{n,\boldepsi}$ is the one of $F_{C_-}$, at most.
Thus, the conclusion easily follows from the next lemma giving the maximum for $F_{C_-}$.
\begin{lemma}\label{lem:A3}
Let $C_{-}$ be the $n$-dimensional scheme generated by the signs $\boldepsi_-\defeq(-1,\cdots,-1)$.
Then
\[
F_{C_{-}}(z_1,\ldots,z_n)\leq 2^{\intpart{\frac{n+1}{2}}}
\qquad \forall (z_1,\ldots,z_n)\in [0,1]^n.
\]
\end{lemma}
\begin{proof}
The graphical scheme $C_{-}$ has the form
\[
\ytableausetup{nosmalltableaux}
\begin{ytableau}
 {-}   & {+}   & {-}   & {+}   & \none[\ \cdots]\\
 \none & {-}   & {+}   & {-}   & \none[\ \cdots]\\
 \none & \none & {-}   & {+}   & \none[\ \cdots]\\
 \none & \none & \none & {-}   & \none[\ \cdots]\\
 \none & \none & \none & \none & \none[\ \cdots]\\
\end{ytableau}
\]
where every row starts with a sign $-$ and continues with alternating signs. We know that the claim for
$n=1$ and $n=2$ is true thanks to~\eqref{eq:A3}.
Let $n\geq 3$. If $n$ is odd, the scheme $C_{-}$ has the form
\[
\ytableausetup{nosmalltableaux}
\begin{ytableau}
 {-}   & {+}   & {-}   & {+}   & \none[\ \cdots \ ]  & {-}  & {+}   & {-}   \\
 \none & {-}   & {+}   & {-}   & \none[\ \cdots \ ]  & {+}  & {-}   & {+}   \\
 \none & \none & \none & \none & \none[\ C_{-,n-2}\ ]  & \none& \none & \none \\
\end{ytableau}
\]
while for $n$ even has the form
\[
\begin{ytableau}
 {-}   & {+}   & {-}   & {+}   & \none[\ \cdots \ ] & {-}   & {+}   \\
 \none & {-}   & {+}   & {-}   & \none[\ \cdots \ ] & {+}   & {-}   \\
 \none & \none & \none & \none & \none[\ C_{-,n-2}\ ] & \none & \none \\
\end{ytableau}
\]
where in both cases $C_{-,n-2}$ is the $n-2$-dimensional scheme defined by the $n-2$-long vector with all
minus signs. By inductive hypothesis, we have $F_{C_{-,n-2}}\leq 2^{\intpart{(n-1)/2}}$.\\
Let us look at the first two rows of $C_{{-}}$: here, the first two columns form a triangular array in
dimension $2$: hence $F_{C_{1,1}}F_{C_{1,2}}F_{C_{2,2}}\leq 2$ by Equation~\eqref{eq:A3}. Moreover, there
are
$\intpart{(n-2)/2}$ consecutive squares %
$\ytableausetup{smalltableaux,centertableaux}
\begin{ytableau}
 {-} & {+}\\
 {+} & {-}\\
\end{ytableau}$,
plus, in case $n$ is odd, an extra vertical segment
$\begin{ytableau}
 {-}\\
 {+}\\
\end{ytableau}$.
Entries~V) and S) of Lemma~\ref{lem:A2} prove that the contribution of each such square and of the
vertical segment are bounded by $1$. Hence, in every case the contribution of the first two rows is
estimated by $2$, and
\[
F_{C_{-}}
\leq 2\cdot F_{C_{-,n-2}}
\leq 2\cdot 2^{\intpart{\frac{n-1}{2}}}
=    2^{\intpart{\frac{n+1}{2}}}.
\]
\end{proof}
To succeed in this task we need to further investigate some properties of the schemes generated by sign
vectors; they are contained in next three lemmas.
\begin{lemma}\label{lem:A4}
Let $C(\boldepsi)$ be a scheme generated by the sign vector $\boldepsi$ of dimension $n\geq 3$. Let $i<
i'$, $j< j'$ with $i'<j$. The product of the four signs $C(\boldepsi)_{i,j}$, $C(\boldepsi)_{i',j}$,
$C(\boldepsi)_{i,j'}$ and $C(\boldepsi)_{i',j'}$ is $1$. In other words, the number of minus signs
in every square %
$\ytableausetup{aligntableaux=bottom}
\begin{ytableau}
 \none     & \none[j] & \none[j']\\
 \none[i]  & { }      & { }\\
 \none[i'] & { }      & { }
\end{ytableau}
\ytableausetup{nosmalltableaux}
$
is even.
\end{lemma}
\begin{proof}
In fact, we have
\begin{align*}
C(\boldepsi)_{i,j} C(\boldepsi)_{i',j} C(\boldepsi)_{i,j'} C(\boldepsi)_{i',j'}
= \prod_{k=i}^j \epsi_k \prod_{k=i'}^j \epsi_k \prod_{k=i}^{j'} \epsi_k \prod_{k=i'}^{j'} \epsi_k
= \prod_{k=j+1}^{j'} \epsi_k \prod_{k=j+1}^{j'} \epsi_k
= 1.
\end{align*}
\end{proof}
Let $C$ be a graphical scheme. We say that the sign $C_{i,j}$ is  \emph{correct} if $C_{i,j} =
(-1)^{i-j+1}$, otherwise we say that $C_{i,j}$ is \emph{wrong}. It is clear that the only graphical
scheme having only correct signs is $C_-$, i.e., the one generated by the signs vector
$\boldepsi_{-}\defeq(-1,\ldots,-1)$.
\begin{lemma}\label{lem:A5}
Let $C(\boldepsi)$ be a scheme generated by the sign vector $\boldepsi$ of dimension $n$ and for $i\leq j
\leq n$ let $H(i,j):=\sum_{u=i}^{j-1} C(\boldepsi)_{i,u}$ (the sum of entries in $C(\boldepsi)$ appearing
to the left of $C(\boldepsi)_{i,j}$), and $V(i,j):=\sum_{v=i+1}^{j} C(\boldepsi)_{v,j}$ (the sum of
entries in $C(\boldepsi)$ appearing below $C(\boldepsi)_{i,j}$). Suppose that $C(\boldepsi)_{i,j}=-$,
then $H(i,j) = -V(i,j)$.
\end{lemma}
\begin{proof}
In fact, $C(\boldepsi)_{i,u} = \prod_{k=i}^u \epsi_k$ and by hypothesis $C(\boldepsi)_{i,j} =
\prod_{k=i}^j \epsi_k = -1$. Thus, for $i\leq u\leq j-1$ we get
\[
C(\boldepsi)_{i,u}
= \prod_{k=i}^u \epsi_k
= - C(\boldepsi)_{i,j}\prod_{k=i}^u \epsi_k
= - \prod_{k=i}^j \epsi_k \prod_{k=i}^u \epsi_k
= - \prod_{k=u+1}^j \epsi_k
= - C(\boldepsi)_{u+1,j}.
\]
Hence, each term appearing below $C(\boldepsi)_{i,j}$ is opposite to a convenient term appearing to the
left of $C(\boldepsi)_{i,j}$, and vice versa.
\end{proof}
We introduce the following quantities, again under the hypothesis that $i\leq j$.
\begin{gather*}
H_{\pm}^w(i,j)\defeq \#\{k\colon i\leq k\leq j-1, C_{i,k}=\pm, C_{i,k} \text{ is wrong}\},\\
V_{\pm}^w(i,j)\defeq \#\{k\colon i+1\leq k\leq j, C_{k,j}=\pm, C_{k,j} \text{ is wrong}\},\\
H^w(i,j)\defeq H_{+}^w(i,j)-H_{-}^w(i,j),
\qquad
V^w(i,j)\defeq V_{+}^w(i,j)-V_{-}^w(i,j).
\end{gather*}
\begin{lemma}\label{lem:A6}
Let $C(\boldepsi)$ be a scheme generated by the sign vector $\boldepsi$ and assume that
$C(\boldepsi)_{i,j}=-$ and that $i+j$ is odd. Then $V(i,j) = 2V^w(i,j)-1$ and $H(i,j)=2H^w(i,j)-1$. We
know that $V(i,j)$ and $H(i,j)$ are opposite in sign by Lemma~\ref{lem:A5}, therefore $H^w(i,j) +
V^w(i,j) = 1$ and in particular, at least one between $H^w(i,j)$ and $V^w(i,j)$ is positive.
\end{lemma}
\begin{proof}
Since $i+j$ is odd, there exist $j-i$ signs $C(\boldepsi)_{l,j}$ below $C(\boldepsi)_{i,j}$, and the
quantity $j-l+1$ is odd for $(j-i+1)/2$ of them, and is even for the remaining $(j-i-1)/2$ cases. Wrong
$+$'s below $C(\boldepsi)_{i,j}$ appear at positions $(l,j)$ where $j-l+1$ is odd, and every other sign
here which is not a wrong $+$ is necessarily a $-$ (actually a correct $-$, but this in not important
now), thus
\[
\sum_{\substack{l=i+1\\j-l+1 \odd}}^j C(\boldepsi)_{l,j}
= V_{+}^w(i,j) - \Big(\frac{1}{2}(j-i+1)-V_{+}^w(i,j)\Big)
= 2V_{+}^w(i,j) - \frac{1}{2}(j-i+1).
\]
Similarly, wrong $-$'s below $C(\boldepsi)_{i,j}$ appear at positions $(l,j)$ where $j-l+1$ is even, and
every other sign here which is not a wrong $-$ is necessarily a $+$, so that
\[
\sum_{\substack{l=i+1\\ j-l+1 \even}}^j C(\boldepsi)_{l,j}
= -V_{-}^w(i,j) + \Big(\frac{1}{2}(j-i-1)-V_{-}^w(i,j)\Big)
= -2V_{-}^w(i,j) + \frac{1}{2}(j-i-1)
\]
Thus
\begin{align*}
V(i,j)
 = \sum_{l=i+1}^j C(\boldepsi)_{l,j}
 = 2(V_{+}^w(i,j)-V_{-}^w(i,j)) - \frac{1}{2}(j-i+1) + \frac{1}{2}(j-i-1)
 = 2V^w(i,j)-1.
\end{align*}
The proof for $H(i,j)$ is similar.
\end{proof}

\section{The procedure}
We are now ready to prove the following theorem. As recalled in the previous section, it yields
Theorem~\ref{th:A1} as immediate corollary thanks to Lemma~\ref{lem:A2} and Lemma~\ref{lem:A3}.
\begin{theorem}\label{th:A2}
Let $C=C(\boldepsi)$ be the scheme generated by any signs vector $\boldepsi$ and let $C_{-}$ be the scheme
generated by the sign vector $\boldepsi$ with all negative signs. There is a list $\mathcal{L}$ of
transformations of type $\Point$, $\Horizontal$, $\Vertical$ and $\Square$ which changes $C$ into
$C_{-}$.
\end{theorem}
\begin{proof}
Let $\boldepsi=(\epsi_1,\ldots,\epsi_n)$ be the signs vector producing $C$. We prove the theorem by
making induction on the dimension $n$.\\
If $n=1$, we only have two possibilities: either %
$C =
\ytableausetup{smalltableaux,aligntableaux=bottom}
\begin{ytableau}
{-}
\end{ytableau}
\ytableausetup{nosmalltableaux}
$ %
and we have finished, or %
$C =
\ytableausetup{smalltableaux,aligntableaux=bottom}
\begin{ytableau}
{+}
\end{ytableau}
\ytableausetup{nosmalltableaux}
$ %
and the claim follows by applying $\Point(1;1)$.\\
Now, assume that $n>1$ and that the claim is true for every scheme generated by any signs pattern of
dimension less than $n$. Let $C'$ be the scheme obtained by removing the $n$-th column from $C$: this is
the scheme generated by the signs vector omitting $\epsi_n$ in $\boldepsi$.
By inductive hypothesis, there exists a list $\mathcal{L}'$ of moves which applied to $C'$ gives
$C_{-}'$, the array of dimension $n-1$ defined by all negative signs.
Our goal is to modify some elements in $\mathcal{L}'$ by replacing them with other moves which correct
all wrong symbol in the $n$-th column and that coincide with the old move on the common part in $C'$: in
this way we will obtain a new list $\mathcal{L}$ of moves that applied to $C$ give $C_{-}$.\\
Moreover, in order to prove that the algorithm can be correctly performed, we need to keep note of each
move we introduce, and of its effect on the $n$-th column. For this purpose we introduce the symbols
$D_{(1)}$, $D_{(2)}$ and so on, to denote the several new versions of the $n$-th column we get after
each new move is performed.
At the beginning we have $D_{(1)}$, which coincides with the $n$-th column in $C$.\\
We start running the column $D_{(k)}$ from the bottom to the top, looking for wrong signs $-$. In case
such signs do not appear, we skip this step and we go directly to the last one. On the contrary, suppose
that we have found a \emph{wrong} $-$ in $i$-th line. We will see that in each new version of the column
only some wrong positions are changed with respect to its previous version. As a consequence, the
\emph{wrong} $-$ in line $i$-th we have detected now was already there at the beginning, i.e.,
$C_{i,n}=-$ and $i+n$ is odd. We compute both $V^w(i,n)$ and $V_{\new}^w(i,n)$, which are the sum of
\emph{wrong} signs appearing below the $(i,n)$ position respectively in $C$, the original scheme, and in
the column $D_{(k)}$: at the beginning evidently numbers $V^w(i,n)$ and $V_{\new}^w(i,n)$ coincide, but
as the algorithm progresses the second may change its value. However, we will check that after each move
we will introduce is executed, the value of the index
\[
\big[\text{number of \emph{wrong} $+$ below $l$ in $n$-th column}\big]
- \big[\text{number of \emph{wrong} $-$ below $l$ in $n$-th column}\big]
\]
for each $l <i$ does not decrease. This proves that the number $V_{\new}^w(i,n)$ we compute in any time
is for sure $\geq V^w(i,n)$.
\smallskip\\
We note that the number $V(i,n)$ is odd, by Lemma~\ref{lem:A6}. In particular, it cannot be $0$.

Suppose that $V(i,n)>0$. Then $V^w(i,n)>0$ by Lemma~\ref{lem:A6}, and $V_{\new}^w(i,n)$ is positive as
well by the previous remark. This means that in some position below $(i,n)$ there is a \emph{wrong} $+$
in column $n$. Let $i'$ be the first (i.e., smallest) index $i'>i$ such that in the $(i',n)$ position
there is a wrong $+$. We add to $\mathcal{L}'$ the move $\Vertical(i,i';n)$: this move is independent of
the other moves, and converts the wrong $-$ and $+$ in those positions into two correct symbols. This
move does not change the value of
\[
\big[\text{number of \emph{wrong} $+$ below $l$ in $n$-th column}\big]
- \big[\text{number of \emph{wrong} $-$ below $l$ in $n$-th column}\big]
\]
for each $l <i$, because the move simply exchanges a $+$ with a $-$ both in positions below the $l$-th
position.
\smallskip

Suppose that $V(i,n)<0$. Then $V^w(i,n) < 0$ and $H^w(i,n)$ is positive, both by Lemma~\ref{lem:A6}.
Thus, in the $i$-th horizontal line to the left of $C_{i,n}$, and hence in $C'$, there is an excess of
\emph{wrong} $+$'s with respect to \emph{wrong} $-$'s. By induction there are moves in $\mathcal{L}'$
changing all these \emph{wrong} entries. Moves of type $\Horizontal$ or $\Square$ cannot be the unique
moves in $\mathcal{L}'$ affecting these positions, since they exchange both a \emph{wrong} $+$ and a
\emph{wrong} $-$ and therefore cannot remove the excess. Also a move of type $\Vertical(i,i';j)$
$\ytableausetup{smalltableaux,aligntableaux=bottom}
\begin{ytableau}
 \none     & \none[j] \\
 \none[i]  & {-}       \\
 \none[i'] & {+}       \\
\end{ytableau}
$ %
is not sufficient to remove the excess, since it removes only a \emph{wrong} $-$ from that line, a
fact which actually increases the excess. Thus, at least a move $\Point(i;j)$
$\ytableausetup{smalltableaux,aligntableaux=bottom}
\begin{ytableau}
 \none     & \none[j]\\
 \none[i]  & {+}      \\
\end{ytableau}
$ %
or a move $\Vertical(i',i;j)$
$\ytableausetup{smalltableaux,aligntableaux=bottom}
\begin{ytableau}
 \none     & \none[j]\\
 \none[i'] & {-}      \\
 \none[i]  & {+}      \\
\end{ytableau}
$ %
is in $\mathcal{L}'$. Let us take $j$ to be the greatest index $<n$ such that this happens. In the first
case we substitute $\Point(i;j)$ with $\Horizontal(i;j,n)$
$\ytableausetup{smalltableaux,aligntableaux=bottom}
\begin{ytableau}
 \none     & \none[j]& \none[n]   \\
 \none[i]  & {+}      & *(\CoTh){-}\\
\end{ytableau}
$ %
which has the same effect on the $C'$ part of the configuration.
In the second case we note that the signs at $(i',j)$, $(i,j)$ and $(i,n)$ positions are
$\ytableausetup{smalltableaux,aligntableaux=bottom}
\begin{ytableau}
 \none     & \none[j]& \none[n]   \\
 \none[i'] & {-}      & { }        \\
 \none[i]  & {+}      & *(\CoTh){-}\\
\end{ytableau}
$. %
By Lemma~\ref{lem:A4} the fourth corner $C_{i',n}$ of the square in $C$ is a \emph{wrong} $+$. We will
show in a moment that this is a $+$ also in $D_{(k)}$, i.e. it appears also at this stage of the
algorithm. Letting this fact for granted for the moment, we proceed substituting $\Vertical(i',i;j)$ in
$\mathcal{L}'$ with $\Square(i',i;j,n)$
$\ytableausetup{smalltableaux,aligntableaux=bottom}
\begin{ytableau}
 \none     & \none[j]& \none[n]   \\
 \none[i'] & {-}      & {+}        \\
 \none[i]  & {+}      & *(\CoTh){-}\\
\end{ytableau}
$ %
which again has the same effect on the $C'$ part of the configuration.
Both moves change
\[
\big[\text{number of \emph{wrong} $+$ below $l$ in $n$-th column}\big]
- \big[\text{number of \emph{wrong} $-$ below $l$ in $n$-th column}\big]
\]
in positions $l< i$. However, the first one actually simply removes a \emph{wrong} $-$, so that it
increases the index for all $l <i$, while the second one increases it when $ i' \leq l < i $ (because
it removes the \emph{wrong} $-$), and keeps unchanged its value for $l < i'$ (because then also the
cancellation of the \emph{wrong} $+$ at $C_{i',n}$ matters).\\
We execute the move we have selected, getting the new column which is $D_{(k+1)}$, by definition.
We repeat this cycle again and again, removing all \emph{wrong} $-$'s from the $n$-column in $C$.
Finally, we add $\Point$ moves to $\mathcal{L}'$ to remove any remaining \emph{wrong} $+$'s in last
column, if any exists.
\smallskip

The description of the algorithm ends here, but we have to resume the point we have skipped before, i.e.,
the proof of the fact that the \emph{wrong} $+$ appearing at the fourth corner $(i',n)$ of the square in
$C$ also appears in $D_{(k)}$, i.e. it appears also at that stage of the algorithm. Suppose the contrary,
i.e., that the \emph{wrong} $+$ is no more there, since it has been corrected at some earlier step of the
algorithm. Then, there had been some index $i''>i$ with $C_{i'',n}=-$ whose correction needed the
substitution of some move $\Vertical(i',i'';j')$ in $\mathcal{L}'$ with $\Square(i',i'';j',n)$ for some
$j'\neq j$, because this is the only possible way the algorithm can correct the $+$ at $(i',n)$ at some
previous step (the case $j'=j$ is for sure impossible, otherwise the \emph{wrong} $-$ at $(i',j')$ would
be corrected in that previous step and would not be available at $k$-th step). This means that we have
one of the following signs patterns in $C$:
\[
\ytableausetup{smalltableaux,aligntableaux=center}
\begin{ytableau}
 \none     & \none[j]& \none[j']& \none[n]  \\
 \none[i'] & {-}      & {-}       & {+} \\
 \none[i]  & {+}      & { }       & {-} \\
 \none[i'']& { }      & {+}       & {-} \\
\end{ytableau}
\quad \text{if $j'>j$, or}\quad
\begin{ytableau}
 \none     & \none[j']& \none[j]& \none[n]  \\
 \none[i'] & {-}       & {-}      & {+} \\
 \none[i]  & { }       & {+}      & {-} \\
 \none[i'']& {+}       & { }      & {-} \\
\end{ytableau}
\quad \text{if $j'<j$.}
\]
In both cases, at $(i,j')$ position we have a \emph{wrong} $+$ (by Lemma~\eqref{lem:A4}, when the square
in positions $(i',j)$, $(i',j')$, $(i,j)$, $(i,j')$ is considered), and the patterns in columns $j'$
and $n$ is
\[
\ytableausetup{smalltableaux,aligntableaux=center}
\begin{ytableau}
 \none     & \none[j']& \none[n]  \\
 \none[i'] & {-}      & {+} \\
 \none[i]  & {+}      & {-} \\
 \none[i'']& {+}      & {-} \\
\end{ytableau}
\]
in both cases. Moreover, $\mathcal{L}'$ contains $\Vertical(i',i'';j')$. However, this is impossible,
since the pattern shows that at $(i,j')$ we have a \emph{wrong} $+$ which is closer to the \emph{wrong}
$-$ at $(i',j')$ than the \emph{wrong} $+$ at $(i'',j')$: this means that when the algorithm has been
applied at an early stage to produce the moves in $\mathcal{L}'$ dealing with the $j'$-th column, we
should have contradicted the prescription according to which every vertical move contains the $+$ which
appears at the closest position to the $-$ in that move.
\end{proof}

An example can be useful to understand the algorithm. Let $C(\boldepsi)$ be the configuration in
dimension $7$ which is generated by signs $\boldepsi=(+,-,+,+,-,+,+)$. Thus,
\[
C(\boldepsi)
=
\ytableausetup{smalltableaux,centertableaux}
\begin{ytableau}
 {+}   & {-}   & {-}   & {-}   & {+}   & {+}   & {+}\\
 \none & {-}   & {-}   & {-}   & {+}   & {+}   & {+}\\
 \none & \none & {+}   & {+}   & {-}   & {-}   & {-}\\
 \none & \none & \none & {+}   & {-}   & {-}   & {-}\\
 \none & \none & \none & \none & {-}   & {-}   & {-}\\
 \none & \none & \none & \none & \none & {+}   & {+}\\
 \none & \none & \none & \none & \none & \none & {+}\\
\end{ytableau}
.
\]
Then, applying the algorithm iteratively, we get:
\begin{align*}
\begin{ytableau}
 {+}\\
\end{ytableau}
&\quad \implies \quad
{\mathcal L}^{(1)} =\{\Point(1;1)\}\\
\begin{ytableau}
 {+}   \\
 \none \\
\end{ytableau}
\
\begin{ytableau}
 *(\CoTh){-}\\
         {-}\\
\end{ytableau}
&\quad \implies \quad
{\mathcal L}^{(2)} =\{\Horizontal(1;1,2)\}\\
\begin{ytableau}
 {+}   & {-}   \\
 \none & {-}   \\
 \none & \none \\
\end{ytableau}
\
\begin{ytableau}
         {-}\\
 *(\CoTh){-}\\
         {+}\\
\end{ytableau}
&\quad \implies \quad
{\mathcal L}^{(3)} =\{\Horizontal(1;1,2), \Vertical(2,3;3)\}\\
\begin{ytableau}
 {+}   & {-}   & {-}   \\
 \none & {-}   & {-}   \\
 \none & \none & {+}   \\
 \none & \none & \none \\
\end{ytableau}
\
\begin{ytableau}
 *(\CoTh){-}\\
         {-}\\
         {+}\\
         {+}\\
\end{ytableau}
&\quad \implies \quad
{\mathcal L}^{(4)} =\{\Horizontal(1;1,2), \Vertical(2,3;3), \Vertical(1,4;4)\}\\
\begin{ytableau}
 {+}   & {-}   & {-}   & {-}   \\
 \none & {-}   & {-}   & {-}   \\
 \none & \none & {+}   & {+}   \\
 \none & \none & \none & {+}   \\
 \none & \none & \none & \none \\
\end{ytableau}
\
\begin{ytableau}
         {+}\\
         {+}\\
         {-}\\
 *(\CoTh){-}\\
         {-}\\
\end{ytableau}
&\quad \implies \quad
{\mathcal L}^{(5)} =\{\Horizontal(1;1,2), \Vertical(2,3;3), \Square(1,4;4,5)\}\\
\begin{ytableau}
 {+}   & {-}   & {-}   & {-}   & {+}   \\
 \none & {-}   & {-}   & {-}   & {+}   \\
 \none & \none & {+}   & {+}   & {-}   \\
 \none & \none & \none & {+}   & {-}   \\
 \none & \none & \none & \none & {-}   \\
 \none & \none & \none & \none & \none \\
\end{ytableau}
\
\begin{ytableau}
         {+}\\
         {+}\\
 *(\CoTh){-}\\
         {-}\\
 *(\CoTh){-}\\
         {+}\\
\end{ytableau}
&\quad \implies \quad
{\mathcal L}^{(6)} =\{\Horizontal(1;1,2), \Square(2,3;3,6), \Square(1,4;4,5), \Vertical(5,6;6)\}\\
\begin{ytableau}
 {+}   & {-}   & {-}   & {-}   & {+}   & {+}   \\
 \none & {-}   & {-}   & {-}   & {+}   & {+}   \\
 \none & \none & {+}   & {+}   & {-}   & {-}   \\
 \none & \none & \none & {+}   & {-}   & {-}   \\
 \none & \none & \none & \none & {-}   & {-}   \\
 \none & \none & \none & \none & \none & {+}   \\
 \none & \none & \none & \none & \none & \none \\
\end{ytableau}
\
\begin{ytableau}
         {+}\\
         {+}\\
         {-}\\
 *(\CoTh){-}\\
         {-}\\
         {+}\\
         {+}\\
\end{ytableau}
&\quad \implies \quad
{\mathcal L}^{(7)} =
\Big\{
\begin{array}{@{}l@{}}   
\Horizontal(1;1,2), \Square(2,3;3,6), \Square(1,4;4,5), \\
\Vertical(5,6;6), \Vertical(4,7;7), \Point(1;7)
\end{array}
\Big\}
\end{align*}

\bibliographystyle{plain}

\end{document}